\documentclass[11pt]{article}

\usepackage[all]{xy}
\usepackage{graphicx}
\usepackage{amssymb,latexsym,amsmath,amsthm, mathrsfs}
\usepackage{setspace}
\usepackage{fullpage}
\usepackage{eso-pic}

\newcommand{\R}{\mathbb R}
\newcommand{\C}{\mathbb C}
\newcommand{\Q}{\mathbb Q}
\newcommand{\Z}{\mathbb Z}
\newcommand{\N}{\mathbb N}

\newtheorem{theorem}{Theorem}[section]
\newtheorem{lemma}[theorem]{Lemma}
\newtheorem{definition}[theorem]{Definition}
\newtheorem{example}[theorem]{Example}

\newtheorem{proposition}[theorem]{Proposition}
\newtheorem{corollary}[theorem]{Corollary}
\newtheorem{conjecture}[theorem]{Conjecture}
\newtheorem{remark}[theorem]{Remark}
\newenvironment{acknowledgement}[1][Acknowledgements]
{\begin{trivlist} \item[\hskip \labelsep {\bfseries #1}]}
{\end{trivlist}}

\begin{document}
\title{Semi-topological Galois theory and the inverse Galois problem}
\author{Hsuan-Yi Liao, Jyh-Haur Teh}
\date{}

\maketitle

\begin{abstract}
We enhance the analogy between field extensions and covering spaces
by introducing the concept of splitting covering which
correspondences to the splitting field in Galois theory. We define
semi-topological Galois groups for Weierstrass polynomials and prove
the existence of a Galois correspondence. This new tool enables us
to study the inverse Galois problem from a new viewpoint.
\end{abstract}

\section{Introduction}
It is well known that there is a Galois correspondence between
subgroups of the fundamental group of a topological space $X$ and
covering spaces over it which is analogous to the Galois
correspondence between field extensions and Galois groups. To make
the analogy stronger, we study some covering spaces defined by some
polynomials in $\mathcal{C}(X)[z]$. We introduce the splitting
coverings for Weierstrass polynomials and define the
semi-topological Galois groups for them. We have the following
correspondences:
$$
\begin{array}{ccc}
fields & \longleftrightarrow & topological \ \ spaces\\
splitting \ \ fields & \longleftrightarrow  & splitting \ \ coverings\\
extension \ \ fields & \longleftrightarrow & covering \ \ spaces\\
Galois \ \ extensions & \longleftrightarrow & Galois \ \ covering \ \ spaces\\
separable \ \ polynomials & \longleftrightarrow & Weierstrass \ \ polynomials\\
Galois \ \ groups & \longleftrightarrow & semi-topological \ \ Galois \ \ groups \\
algebraic \ \ closures & \longleftrightarrow & universal \ \ coverings\\
\end{array}$$

A Weierstrass polynomial $f$ in $\mathcal{C}(X)[z]$ is a polynomial
such that each root is of degree one. We construct its splitting
covering $p:E_f \rightarrow X$ and show that it is smallest among
covering spaces that $f$ splits. We define the semi-topological
Galois group of $f$ and show that it is isomorphic to the group of
covering transformations of $E_f \rightarrow X$. The proof uses a
result of Chase-Harrison-Rosenberg about the existence of Galois
correspondence in commutative rings. In the final section, we use
semi-topological Galois groups to study the inverse Galois problem.
We gives some criterion such that a finite group can be realized as
the Galois group of some field extension over $\Q(i)$.

\section{Algebraic closures and universal coverings}
Throughout this article, unless otherwise stated, $X,Y,Z$ will
denote topological spaces which are Hausdorff, path-connected,
locally path-connected and semi-locally simply connected. Let $
\mathcal{C}(X)$ be the ring of all continuous functions from $X$ to
$\C$ and $f=f_{x}(z)=a_{n}(x)z^{n}+a_{n-1}(x)z^{n-1}+\cdots
+a_{0}(x)$ be an element in $\mathcal{C}(X)[z]$, the polynomial ring
with coefficients in $\mathcal{C}(X)$. In general, there may no
exist a continuous function $\alpha:X \rightarrow \C$ such that
$f_x(\alpha(x))=0$ for all $x\in X$. For example, on the unit sphere
$S^1$, there is no continuous function in $\mathcal{C}(S^1)$ which
satisfies the equation $z^2-x=0$.

\begin{definition}
Let $n \in \N$ and $f=f_{x}(z)=z^{n}+a_{n-1}(x)z^{n-1}+\cdots
+a_{0}(x) \in \mathcal{C}(X)[z]$. $f$ is called a
\textbf{Weierstrass polynomial} of degree $n$ on $X$ if for each $x
\in X$, $f_{x}$ has distinct $n$ roots. For such $f$, $E=\lbrace
(x,z) \in X \times \C : f_{x}(z)=0  \rbrace$ is called the
\textbf{solution space} of $f$. A \textbf{root} of $f$ is a continuous
function $\alpha: X \rightarrow \C$ such that $f_x(\alpha(x))=0$ for
all $x\in X$. We say that $f$ \textbf{splits} in $X$ if $f$ has $n$
distinct roots in $X$. A Weierstrass polynomial is called
\textbf{irreducible} if it is irreducible as an element in the ring
$\mathcal{C}(X)[z]$.
\end{definition}

It is well known that the solution space of a Weierstrass polynomial
under the first projection is a covering space over $X$, and the
solution space of a Weierstrass polynomial is connected if and only
if the Weierstrass polynomial is irreducible(\cite[Theorem 4.2, pg 141]{Ha}).

Since a Weierstrass polynomial $f\in \mathcal{C}(X)[z]$ may no have
solutions in $X$, it is natural to ask if we can find solutions of
$f$ in some covering spaces over $X$. This is analogous to finding
roots of a polynomial in some field extensions in Galois theory. We
will see soon that the universal cover of $X$ plays the role of
algebraic closure.

\begin{definition}
Let $\lambda:Y \rightarrow X$ be a continuous map. The \textbf{pullback}
$\lambda^{*}:\mathcal{C}(X) \rightarrow \mathcal{C}(Y)$ is defined
by
$$\lambda^{*}(\gamma):=\gamma \circ \lambda$$ which induces a ring
homomorphism $\lambda^{*}:\mathcal{C}(X)[z] \rightarrow
\mathcal{C}(Y)[z]$ by
$$\lambda^{*}(a_{n}z^{n}+a_{n-1}z^{n-1}+\cdots +a_{0}):=(a_{n}\circ\lambda
)z^{n}+(a_{n-1}\circ\lambda)z^{n-1}+\cdots +(a_{0}\circ\lambda).$$
\end{definition}

The following basic facts about covering spaces can be found in
\cite{GH} or \cite{Mu}. We quote them here since we need to use them often and
for the convenience of the reader.

\begin{theorem}\label{lift}
Let $X,Y,E$ be connected and locally path-connected topological spaces and
$x_{0} \in X, y_{0} \in Y, e_{0} \in E$. Suppose $(E,e_{0})\stackrel{p}{\rightarrow} (X,x_{0})$
is a covering space and $f:(Y,y_{0}) \rightarrow (X,x_{0})$ is a continuous map.
Then there exists a lifting $f^{'}$ of $f$ ,that is, $pf^{'}=f$ if and only if
$$f_{*}\pi_{1}(Y,y_{0}) \subset p_{*}\pi_{1}(E,e_{0}).$$ In particular, if $Y$
is simply connected, then $f$ always has a lifting.
\end{theorem}

\begin{theorem}(Unique lifting property)\label{ULT}
Let $(E,e_{0})\stackrel{p}{\rightarrow} (X,x_{0})$ be a covering
space with based points and $f:(Y,y_{0}) \rightarrow (X,x_{0})$ be a
continuous map. Assume $Y$ is connected. If there exists a
continuous map $f^{'}:(Y,y_{0})\rightarrow (E,e_{0})$ such that
$pf^{'}=f$, then it is unique.
\end{theorem}

\begin{proposition}
If $f$ is a Weierstrass polynomial of degree $n$ and
$Y\stackrel{p}{\rightarrow} X$ is a connected covering space, then
any two roots of $p^{*}f$ are either equal everywhere or equal nowhere; in particular, $p^{*}f$ has at most $n$ roots.
\end{proposition}

\begin{proof}
    Suppose $\alpha , \beta$ are roots of $p^{*}f$. Let $A=\lbrace y \in Y \mid \alpha(y)=\beta(y)\rbrace$. Then $A$ is closed in $Y$ as $\C$ is Hausdorff. Assume $E$ is the solution space of $f$, and $pr_{1}$, $pr_{2}$ are the first and second projections respectively. For $y \in A$, there is a neighborhood $U$ of $p(y)$ such that $pr_{1}^{-1}(U)=\coprod_{i=1}^{n}U_{i}$ is a trivial covering on $U$. Since $\C$ is Hausdorff, we may take a smaller neighborhood if necessary such that $pr_{2}(U_{1}),\cdots ,pr_{2}(U_{n})$ lie in some disjoint open subsets $V_{1}, \cdots ,V_{n}$ of $\C$ respectively. $\alpha(y)=\beta(y)$ lies in one of $V_{1}, \cdots ,V_{n}$. Assume $\alpha(y)=\beta(y) \in V_{1}$. Then $ W=\alpha^{-1}(V_{1}) \cap \beta^{-1}(V_{1}) \cap p^{-1}(U)$ is an open neighborhood of  $y$ in $Y$ and $W \subset A$. Hence $A$ is open in $Y$. Consequently, $A$ is empty or whole $Y$. In other words, two roots of $p^{*}f$ are either equal everywhere or equal nowhere; thus, $p^{*}f$ has at most $n$ roots.
\end{proof}

\begin{theorem}\label{sol} (Algebraic closure)
Let $f$ be a Weierstrass polynomial on $X$ of degree $n$. Then $f$
splits in $\widetilde{X}$ where $p:(\widetilde{X}, \widetilde{x}_0)
\rightarrow (X, x_0)$ is the universal covering of $X$.
\end{theorem}

\begin{proof}
Let $E_{1}, ... ,E_{k}$ be all path-connected components of the
solution space $\pi:E \rightarrow X$ of $f$. Then for each $i$,
$\pi_{i}:=\pi\vert_{E_{i}}:E_{i} \rightarrow X$ is a covering space
of $X$. Let
$$(\pi_{i})^{-1}(x_{0}) = \lbrace e_{i,1},... ,e_{i,r_{i}} \rbrace .$$
From Theorem~\ref{lift} and the unique lifting theorem, for each
$e_{i,j}$, there exists a unique lifting,
$\tilde{p}_{i,j}:(\tilde{X},\tilde{x}_{0}) \rightarrow
(E_{i},e_{i,j})$ of $p$. Define $\alpha_{i,j}:=q_{i} \circ
\tilde{p}_{i,j}$ which are roots of $p^*f$ where $q_{i}: E_{i}
\rightarrow \C$ is the projection to the second factor. We have the
following commutative diagram
$$\xymatrix{
\tilde{X} \ar[dd]_{p} \ar[rd]_{\tilde{p}_{i,j}} \ar[rrrd]^{\alpha_{i,j}} & & & \\
& E_{i} \ar[ld]^{\pi_{i}} \ar[rr]^{q_{i}} &  & \C \\
X & & & }$$ Note that if $(i,j) \neq (i^{'},j^{'})$, then
$q_{i}(e_{i,j}) \neq q_{i^{'}}(e_{i^{'},j^{'}})$. Hence
$\alpha_{i,j}(\tilde{x}_{0})\neq
\alpha_{i^{'},j^{'}}(\tilde{x}_{0})$. Since $r_{1}+r_{2}+\cdots
+r_{k}=n$, the maps
$$\alpha_{i, j}, j=1, ... , r_i, i=1, ... , k$$
are all the roots of $p^*f$.
\end{proof}

\section{Splitting coverings}
Let $Y \stackrel{p}{\rightarrow} X$ be a covering space of $X$. We
denote the group of covering transformations by $A(Y/X)$, that is,
$$A(Y/X)=\lbrace \Phi:Y\rightarrow Y \mid \Phi \; is \; a \;
homeomorphism \; such \; that \; p\Phi=p \rbrace.$$

\begin{definition}
Let $Y \stackrel{p}{\rightarrow} X$ be a covering space of $X$. $Y
\stackrel{p}{\rightarrow} X$ is called a \textbf{Galois covering} if
$A(Y/X)$ acts on a fibre of $X$ transitively (hence all fibres).
\end{definition}

\begin{remark}
From theory of covering spaces (\cite[pg 25]{GH}), the above
definition is equivalent to $$p_{*}\pi_{1}(Y,y_{0}) \lhd
\pi_{1}(X,x_{0}),$$ where $x_{0} \in X$, $y_{0} \in Y$, and
$p(y_{0})=x_{0}$. In particular, the universal covering is a Galois
covering.
\end{remark}

\subsection{The existence and uniqueness of splitting coverings}

\begin{definition}
Let $f$ be a Weierstrass polynomial of degree n on $X$ and $Y
\stackrel{p}{\rightarrow} X$ be a covering space where $Y$ is
path-connected. $Y$ is said to be a \textbf{splitting covering} of
$f$ if
\begin{enumerate}
\item $f$ splits in $Y$,
\item $Y$ is the smallest among such coverings, that is, if $Y^{'} \stackrel{p^{'}}{\rightarrow} X$
is another covering where $f$ splits, then there exists a covering map $\pi : Y^{'} \rightarrow Y$
such that the diagram $$\xymatrix{
Y^{'} \ar[rr]^{\pi} \ar[rd]_{p^{'}}& & Y \ar[ld]^{p} \\
& X & }$$ commutes.
\end{enumerate}
\end{definition}

Construction of a splitting covering of $f$: Let $h_{0}$ be an
irreducible component of $f$ in $\mathcal{C}(X)[z]$. Let
$E_{1}\stackrel{p_{1}}{\rightarrow} X$ be the solution space of
$h_{0}$ and $\pi_{1}:E_{1} \rightarrow \C$ be the projection to the
second component. Hence
$$(p_{1}^{*}f)(z)=(z-\pi_{1})g_{1}$$ in
$(p_{1}^{*}\mathcal{C}(X))[z]$. Inductively, assume for $i < n$, we
have
$$(p_{i}^{*}f)(z)=(z-q_{i}^{*}\cdots q_{2}^{*}\pi_{1})\cdots (z-q_{i}^{*}\pi_{i-1})(z-\pi_{i})g_{i}$$
in $(p_{i}^{*}\mathcal{C}(X))[z]$, where $$\xymatrix{
E_{i} \ar[r]^{q_{i}} \ar[rrrd]^{p_{i}} & E_{i-1} \ar[r]^{q_{i-1}}& \cdots \ar[r]^{q_{2}}& E_{1} \ar[d]^{p_{1}} \\
& & & X }$$ and $\pi_{j}:E_{j} \rightarrow \C$ is the projection of
the last component, $j=1, ... ,i$. Note that $g_{i}$ is a
Weierstrass polynomial in $E_{i}$. For $i+1$, let $h_{i}$ be an
irreducible component of $g_{i}$ in $(p_{i}^{*}\mathcal{C}(X))[z]$,
$E_{i+1}\stackrel{q_{i+1}}{\rightarrow} E_{i}$ be the solution space
of $h_{i}$, $p_{i+1}=p_{1}q_{2}\cdots q_{i+1}$ and
$\pi_{i+1}:E_{i+1} \rightarrow \C$ be the projection of the last
component. Hence
$$(p_{i+1}^{*}f)(z)=(z-q_{i+1}^{*}\cdots q_{2}^{*}\pi_{1})\cdots (z-q_{i+1}^{*}\pi_{i})(z-\pi_{i+1})g_{i+1}$$
in $(p_{i+1}^{*}\mathcal{C}(X))[z]$. By induction, we have
$E_{f}:=E_{n}\stackrel{q:=p_{n}}{\longrightarrow} X$ and
$$(q^{*}f)(z)=(z-\alpha_{1})\cdots (z-\alpha_{n-1})(z-\alpha_{n})$$ in
$(q^{*}\mathcal{C}(X))[z]$ where
$$\alpha_{j}:=q_{n}^{*}\cdots q_{j+1}^{*}\pi_{j},$$ $j=1,... ,n$. Hence
$E_{f}$ is connected, and $f$ splits in $E_{f}$. Note that an
element in $E_{f}$ is of the form $(\cdots (((x,z_{1}),z_{2})\cdots
,z_{n})$. We identify it as $(x,z_{1},... ,z_{n})$. Then
$\alpha_{j}$ becomes the projection of the $(j+1)^{th}$ component,
$q$ becomes the projection of the first component and $$E_{f}
\subset S_{f}$$ where $$S_{f}:= \lbrace (x,z_{1},... ,z_{n}) \in X
\times \C^{n} : f_{x}(z_{i})=0, \; i=1,... ,n, \; and \; z_{i} \neq
z_{j} \; if \; i \neq j \rbrace.$$ It is clear that
$(pr_{1},pr_{2}):E_{f} \rightarrow E$ is a covering map where $pr_1,
pr_2$ are the projection to the first and second factor
respectively.

\begin{theorem}\label{split}
Let $f$ be a Weierstrass polynomial of degree $n$ in $X$.
\begin{enumerate}
\item $E_{f} \stackrel{q}{\rightarrow} X$ is a splitting covering.
\item Splitting covering is unique up to covering isomorphisms.
\item $E_{f} \stackrel{q}{\rightarrow} X$ is a Galois covering.
\end{enumerate}
\end{theorem}

\begin{proof}
\begin{enumerate}
\item Let $Y \stackrel{p}{\rightarrow} X$ be a covering space such that $f$ splits in $Y$
with roots $\beta_{1},\cdots ,\beta_{n}$. Let $x_{0} \in X$,
$(x_{0},z_{0,1},\cdots ,z_{0,n}) \in E_{f}$ and $y_{0}$ be any
element in $p^{-1}(x_{0})$. After reordering $\beta_{1},\cdots
,\beta_{n}$ if necessary, we can assume
$\beta_{1}(y_{0})=z_{0,1},\cdots ,\beta_{n}(y_{0})=z_{0,n}$.
Define $\pi : Y \rightarrow E_{f}$ by $$\pi
(y)=(p(y),\beta_{1}(y),\cdots ,\beta_{n}(y)).$$ Then
\begin{align} p=q \circ \pi . \tag{1}
\end{align}
For any $e=(x,z_{1},\cdots ,z_{n}) \in E_f$ there exists a
path-connected open neighborhood $U$ of $x$ in $X$ such that
$q^{-1}(U)=\coprod_{i}U_{i}$, $p^{-1}(U)=\coprod_{j}V_{j}$, all
$U_{i}$ and $V_{j}$ are open in $E_{f}$ and $Y$ respectively,
and $q\vert_{U_{i}}$ and $p\vert_{V_{j}}$ are homeomorphisms.
Since $V_{j}$ is path-connected and by (1), $\pi (V_{j}) \subset
U_{i}$ for some $i$. Hence
\begin{align}
\pi\vert_{V_{j}}=(q\vert_{U_{i}})^{-1}\circ (p\vert_{V_{j}})
\tag{2}
\end{align}
is a homeomorphism. Therefore, if $\pi (Y) \cap U_{i} \neq \phi$,
then $\pi (V_{j}) \cap U_{i} \neq \phi$ for some $j$, and hence $U_{i}=\pi(V_{j}) \subset \pi (Y)$.
In other words, either $U_{i} \subset \pi(Y)$ or $U_{i} \subset \pi(Y)^{c}$. Therefore, $\pi(Y)$
is open and closed in $E_{f}$. Since $(x_{0},z_{0,1},\cdots ,z_{0,n}) \in \pi(Y)$ and $E_{f}$
is connected, $\pi$ is surjective. Hence by (2), $Y \stackrel{\pi}{\rightarrow} E_{f}$ is a covering space.

\item Let $x_{0} \in X$ and $e_{0} \in q^{-1}(x_{0})$ Suppose $Y \stackrel{p}{\rightarrow} X$
is also a splitting covering. Then by the proof of part one, there exists a covering map
$\pi:E_{f}\rightarrow Y$ and a covering map $\pi':Y \rightarrow
E_f$ such that the diagram
$$\xymatrix{E_{f} \ar[rr]^{\pi}
\ar[rd]_{q}  & & Y \ar @/^/[ll]^{\pi'}  \ar[ld]^{p} \\& X & }$$
commutes, and $\pi'(\pi(e_{0}))=e_{0}$. By the unique lifting theorem,
$$\pi\circ \pi^{'}=id_{Y}, \; \pi^{'}\circ\pi=id_{E_{f}}.$$
Hence the coverings $Y \stackrel{p}{\rightarrow} X$ and $E_{f}
\stackrel{q}{\rightarrow} X$ are isomorphic.

\item Let $x_{0} \in X$ and $e_{0}, e_{1} \in q^{-1}(x_{0})$.
By the argument in the first part, we have $\pi_{0}:(E_{f},e_{0}) \rightarrow (E_{f},e_{1})$
and $\pi_{1}:(E_{f},e_{1}) \rightarrow (E_{f},e_{0})$ such that $$q=q \circ \pi_{i}, \; i=0,1.$$
By the unique lifting theorem, $$\pi_{1} \circ \pi_{0}=\pi_{0}\circ\pi_{1}=id_{E_{f}},$$ and
hence $\pi_{0} \in A(E_{f}/X)$. Therefore, $A(E_{f}/X)$ acts transitively on $q^{-1}(x_{0})$.
\end{enumerate}
\end{proof}

Recall that in Galois theory, the splitting field of a separable
polynomial is Galois over the base field. The above result makes a
parallel correspondence between splitting fields and splitting
coverings.

\begin{proposition}
Let $f$ be an irreducible Weierstrass polynomial on $X$ of degree
$n$ with solution space $E\stackrel{\pi}{\rightarrow} X$. Suppose
$p: Y\rightarrow E$ is a covering space and $q:Y \rightarrow X$ is a
Galois covering of $X$ where $q=\pi\circ p$. Then $f$ splits in $Y$.
\end{proposition}

\begin{proof}
Let $x_{0} \in X$, $\pi^{-1}(x_{0})=\lbrace e_{1},\cdots
,e_{n}\rbrace$ and $y_{0} \in p^{-1}(e_{1})$. Assume $\gamma_{i}$ is
a path from $e_{1}$ to $e_{i}$ in $E$. Since
$q_{*}\pi_{1}(Y,y_{0})=\pi_{*}p_{*}\pi_{1}(Y,y_{0}) \subset
\pi_{*}\pi_1(E,e_{1})$ and $q_{*}\pi_{1}(Y,y_{0}) \triangleleft
\pi_{1}(X,x_{0})$,
$$q_{*}\pi_{1}(Y,y_{0})=[\pi_{*}\gamma_{i}]q_{*}\pi_{1}(Y,y_{0})[\pi_{*}\gamma_{i}]^{-1}
\subset
[\pi_{*}\gamma_{i}]\pi_{*}\pi_1(E,e_{1})[\pi_{*}\gamma_{i}]^{-1}
=\pi_{*}\pi_1(E,e_{i}).$$ Consequently, there are $p_{i}:(Y,y_{0})
\rightarrow (E,e_{i})$ such that
$$\xymatrix{
(Y,y_{0}) \ar[dd]_{q} \ar[rd]_{p_{i}} \ar[rrrd] & & & \\
& (E,e_{i}) \ar[ld]^{\pi} \ar[rr]_{pr_{2}} &  & \C \\
(X,x_{0}) & & & }$$ commutes. Therefore, $q^{*}f=(z-pr_{2}\circ
p_{1})\cdots (z-pr_{2}\circ p_{n})$ where $pr_2: E \rightarrow \C$
is the projection to the second factor.
\end{proof}

\begin{corollary}\label{CorSplit}
Let $f$ be an irreducible Weierstrass polynomial of degree $n$ on $X$ and
$E \stackrel{\pi}{\rightarrow} X$ be its solution space. Suppose $E \stackrel{\pi}{\rightarrow} X$
is a Galois covering. Then $E \stackrel{\pi}{\rightarrow} X$ is a splitting covering.
\end{corollary}

\subsection{Another construction of splitting coverings}

Recall that any symmetric polynomial in $n$ variables can be written
as a unique polynomial in the elementary symmetric polynomials,
$s_{0},\cdots ,s_{n-1}$  where
$$\prod_{i=1}^{n}(z-z_{i})=z^{n}+\sum_{i=0}^{n-1}(-1)^{n-i}s_{i}(z_{1},\cdots ,z_{n})z^{i}.$$
Hence there is a unique polynomial in $n$ variables $\delta
(a_{0},\cdots ,a_{n-1})$ such that
$$\delta(-s_{n-1}(z_{1},\cdots ,z_{n}),\cdots ,(-1)^{n-i}s_{i}(z_{1},\cdots ,z_{n}),\cdots ,(-1)^{n}s_{0}(z_{1},\cdots ,z_{n}))=\prod_{1\leqslant
i<j \leqslant n}(z_{i}-z_{j}).$$ The polynomial $\delta
(a_{0},\cdots ,a_{n-1})$ is called the \textbf{discriminant
polynomial}. Define $$B^{n}:=\C^{n}-Z(\delta)$$ where $Z(\delta)$ is
the set of zeros of $\delta$.

\begin{lemma}
\begin{enumerate}
\item
Let
$$S:= \lbrace (a_{0},\cdots ,a_{n-1},z_{1},\cdots ,z_{n}) \in
B^{n} \times \C^{n} : z^{n}+\sum_{i=0}^{n-1}a_{i}z^{i} =
\prod_{i=1}^{n}(z-z_{i}) \rbrace$$ and $\pi$ be the projection
to $B^{n}$. Then $\pi:S\rightarrow B^{n}$ is an $n!$-fold
covering space.

\item
Let $S_{f}:= \lbrace (x,z_{1},\cdots ,z_{n}) \in X \times \C^{n}
: f_{x}(z_{i})=0, \; i=1,\cdots ,n, \; and \; z_{i} \neq z_{j}
\; if \; i \neq j \rbrace$. Then $S_{f}
\stackrel{\tilde{q}}{\rightarrow} X$ is an $n!$-fold covering
space where $\tilde{q}$ is the projection to $X$.
\end{enumerate}
\end{lemma}

\begin{proof}
\begin{enumerate}
\item Similar to \cite[pg 88, Lemma 2.2]{Ha}.

\item
Let $f=z^{n}+\sum_{i=0}^{n-1}a_{i}z^{i}$ and $a:X \rightarrow
B^{n}:$ $$a(x):=(a_{0}(x),\cdots ,a_{n-1}(x)).$$ Then we get the
induced fibre bundle
$$\xymatrix{a^{*}(S) \ar[r]^{a^{*}} \ar[d]_{\pi^{*}} & S \ar[d]^{\pi} \\
X \ar[r]^{a} & B^{n}. }$$ Define $a':S_{f} \rightarrow a^{*}(S)$ by
$$a'(x,z_{1},\cdots ,z_{n}):=(x,a(x),z_{1},\cdots ,z_{n}).$$ Then $a'$
is a homeomorphism such that the diagram $$\xymatrix{
S_{f}\ar[rr]^{a'} \ar[rd]_{\tilde{q}} & & a^{*}(S) \ar[ld]^{\pi^{*}} \\
& X &
}$$ commutes. Hence $S_{f} \stackrel{\tilde{q}}{\rightarrow} X$ is an $n!$-fold covering space.
\end{enumerate}
\end{proof}

\begin{proposition}
Let $E'$ be a connected component of $S_{f}$ and
$q:=\tilde{q}\vert_{E'}$. Then $E'\stackrel{q}{\rightarrow} X$ is a
splitting covering.
\end{proposition}

\begin{proof}
By the previous lemma, $E'\stackrel{q}{\rightarrow} X$ is a covering
space. Moreover, $f$ splits in $E'$ with roots $\alpha_{1},\cdots
,\alpha_{n}$ where $\alpha_{i}$ is the projection of the
$(i+1)^{th}$ component, $i=1,\cdots ,n$. The result follows as in
the proof of Theorem~\ref{split}.
\end{proof}

Observe that $S \stackrel{\pi}{\rightarrow} B^{n}$ is a Galois
covering since the map
$(a=(a_{0},\cdots ,a_{n-1}),z_{1},\cdots ,z_{n})\mapsto
(a,z_{\sigma(1)},\cdots ,z_{\sigma(n)})$ is a covering
transformation for each $\sigma \in S_{n}$. Consequently, $S \stackrel{\pi}{\rightarrow} B^{n}$
becomes a locally trivial principal $A(S/B^{n})$-bundle
($A(S/B^{n})$ with discrete topology). Let's recall a classical
result about principal G-bundles (\cite[pg 51, Theorem 9.9]{Hu}).

\begin{proposition}
Let $\xi$ be a numerable principal G-bundle over $B$ (which is true
whenever $B$ is Hausdorff and paracompact), and let $f_{t}:B'
\rightarrow B$ be a homotopy. Then the principal G-bundles
$f_{0}^{*}(\xi)$ and $f_{1}^{*}(\xi)$ are isomorphic over
$B'$.
\end{proposition}

\begin{corollary}\label{homotopy}
Let $f=z^{n}+\sum_{i=0}^{n-1}a_{i}z^{i}$ and
$g=z^{n}+\sum_{i=0}^{n-1}b_{i}z^{i}$ be two Weierstrass polynomials
on a Hausdorff and paracompact space $X$ and $a,b:X \rightarrow
B^{n}:$ $$a(x):=(a_{0}(x),\cdots ,a_{n-1}(x)), \;
b(x):=(b_{0}(x),\cdots ,b_{n-1}(x)).$$ Let
$E_{a}\stackrel{q_{a}}{\rightarrow} X$ and
$E_{b}\stackrel{q_{b}}{\rightarrow} X$ be the splitting cover of $f$
and $g$ respectively. If $a$ and $b$ are homotopic as maps of $X$
into $B^{n}$, then $E_{a}\stackrel{q_{a}}{\rightarrow} X$ and
$E_{b}\stackrel{q_{b}}{\rightarrow} X$ are equivalent covering
spaces.
\end{corollary}

\section{Semi-topological Galois groups}
All rings are assumed to be commutative rings with identity if
without mentioned explicitly.

\begin{definition}
Let $\bar{T}$ be a ring and $T$ be a subring of $\bar{T}$. Define
\begin{align*}
Aut_{T}(\bar{T})&= \lbrace \phi :\bar{T} \rightarrow \bar{T} \mid \phi \; is \; a \; T-algebra \; automorphism \rbrace \\
&= \lbrace \phi :\bar{T} \rightarrow \bar{T} \mid \phi \; is \; a \; ring \; automorphism \; such \; that \; \phi(x)=x, \; \forall x \in T \rbrace.
\end{align*}

Let $f=f_{x}(z)=z^{n}+a_{n-1}(x)z^{n-1}+\cdots +a_{0}(x)$ be a
Weierstrass polynomial on $X$, and let
$E_{f}\stackrel{q}{\rightarrow} X$ be the splitting covering. We
define $$R=q^{*}\mathcal{C}(X)= \lbrace \gamma \circ q : \gamma \in
\mathcal{C}(X) \rbrace,$$ which is a subring of
$\mathcal{C}(E_{f}).$ The \textbf{semi-topological Galois group} of
$f$ is defined to be
$$G_{f}:=Aut_{R}(R[\alpha_{1},\cdots ,\alpha_{n}])$$ where
$\alpha_{1},\cdots ,\alpha_{n}:E_{f}\rightarrow \C$ are the
solutions of $f$.
\end{definition}

    Although the domain of solutions is mentioned to be the splitting covering in the definition of semi-topological group of $f$, it is, in fact, not important. To be precise, we have the following proposition.

\begin{proposition}\label{indep}
    Let $E_{f}\stackrel{q}{\rightarrow} X$ be the splitting covering of $f$ with roots $\alpha_{1}, \cdots , \alpha_{n}: E_{f}\rightarrow \C$. Suppose $f$ splits on $Y \stackrel{p}{\rightarrow}X$ with solutions $\alpha_{1}',\cdots ,\alpha_{n}':Y \rightarrow \C$.
Then there exists an isomorphism $\Phi :q^{*}\mathcal{C}(X)[\alpha_{1},\cdots ,\alpha_{n}] \rightarrow p^{*}\mathcal{C}(X)[\alpha_{1}',\cdots ,\alpha_{n}']$
such that $\Phi(q^{*}\mathcal{C}(X))=p^{*}\mathcal{C}(X)$, $\Phi(\lbrace \alpha_{1},\cdots ,\alpha_{n}\rbrace)=\lbrace \alpha_{1}',\cdots ,\alpha_{n}'\rbrace$,
and hence, $$ G_{f}\cong Aut_{p^{*}\mathcal{C}(X)}p^{*}\mathcal{C}(X)[\alpha_{1}',\cdots ,\alpha_{n}'].$$
\end{proposition}

\begin{proof}
    By the definition of splitting coverings, there is a covering map $\pi :Y \rightarrow E_{f}$ such that $$\xymatrix{
  Y \ar[dr]_{p} \ar[rr]^{\pi} & & E_{f} \ar[ld]^{q} \\
 & X &  }$$ commutes.
Observe that $\pi^{*}\alpha_{1},\cdots,\pi^{*}\alpha_{n}$ are all roots as for $i=1,\cdots,n$, $ (p^{*}f)(\pi^{*}\alpha_{i}) =(\pi^{*}q^{*}f)(\pi^{*}\alpha_{i})=\pi^{*}((q^{*}f)(\alpha_{i}))=0$
and $\pi^{*}\alpha_{1},\cdots,\pi^{*}\alpha_{n}$ are distinct.
Therefore, $\Phi := \pi^{*}:q^{*}\mathcal{C}(X)[\alpha_{1},\cdots,\alpha_{n}]\rightarrow p^{*}\mathcal{C}(X)[\alpha_{1}',\cdots,\alpha_{n}']$ is an isomorphism which carries $q^{*}\mathcal{C}(X)$ onto $p^{*}\mathcal{C}(X)$ and $\lbrace \alpha_{1},\cdots,\alpha_{n} \rbrace$ to $\lbrace \alpha_{1}',\cdots,\alpha_{n}' \rbrace$. As a result, we obtain an isomorphism $\Psi : G_{f} \rightarrow Aut_{p^{*}\mathcal{C}(X)}p^{*}\mathcal{C}(X)[\alpha_{1}',\cdots,\alpha_{n}']$ which is defined by $\Psi(\phi)(g)=(\pi^{*})\phi((\pi^{*})^{-1}g).$
\end{proof}

\begin{example}\label{ExS}
Let $f_x(z)=z^{n}-x$ for $x \in S^{1}$ where  $n \in \N$. Then $f$
is a Weierstrass polynomial, and its solution space $E$ is an n-fold
covering of $S^{1}$.

Let $p:\R \rightarrow X$  be defined by $p(s)=e^{2\pi si}$ which is
the universal covering space of $S^{1}$.
$(p^{*}f)_s(z)=z^{n}-e^{2\pi si}$, where $s \in \R$. It is easy to
see that roots of $p^*f$ are $\alpha_{j}(s)=e^{\frac{2\pi
i(s+j-1)}{n}}, \; j=1,\cdots ,n$. Note that for $j=1,\cdots ,n-1$,
$e^{\frac{2\pi i}{n}}\alpha_{j}=\alpha_{j+1}$, and the constant
function $e^{\frac{2\pi i}{n}}$ is an element in
$R=p^{*}\mathcal{C}(S^{1}).$ Therefore, for $\phi \in G_{f}$,
$$e^{\frac{2\pi i}{n}}\phi(\alpha_{j})=\phi(e^{\frac{2\pi
i}{n}}\alpha_{j})=\phi(\alpha_{j+1}).$$ Hence $\phi$ is uniquely
determined by $\phi(\alpha_{1})$. Let $\sigma :R[\alpha_{1},\cdots
,\alpha_{n}] \rightarrow R[\alpha_{1},\cdots ,\alpha_{n}]$ be
defined by
$$\sigma(\tilde{\gamma})(s):=\tilde{\gamma}(s+1).$$  Then for
$p^{*}\gamma \in R$, $$\sigma(p^{*}\gamma)(s)=\sigma(\gamma \circ
p)(s)=\gamma (p(s+1))=\gamma (p(s))=(p^{*}\gamma)(s).$$ Hence
$\sigma\vert_{R}=id_{R}$. For $j=1,\cdots ,n-1$,
$$\sigma(\alpha_{j})=\alpha_{j+1}, \sigma(\alpha_{n})=\alpha_{1}.$$ Therefore $\sigma \in G_{f}$.
Furthermore, for $j=0,1,\cdots ,n-1$,
$\sigma^{j}(\alpha_{1})=\alpha_{1+j}$, and $\sigma^{n}=id$. Thus,
from the above observations, we have $$G_{f} \cong <\sigma>
\cong \Z_{n}.$$
\end{example}

\begin{proposition}(Functoriality)
Suppose that $\lambda:Y \rightarrow X$ is a covering map and $f_1$ is a
Weierstrass polynomial of degree $n$ in $X$. Let $f_2=\lambda^*f_1$.
\begin{enumerate}
\item
There is a covering map $\widetilde{\lambda}:E_{f_2} \rightarrow
E_{f_1}$ such that the following diagram commutes:
$$\xymatrix{E_{f_2} \ar[r]^{\widetilde{\lambda}} \ar[d] _{q} & E_{f_1}
\ar[d]^p\\
Y \ar[r]^{\lambda} & X}$$

\item
If $\alpha_1, \cdots , \alpha_n:E_{f_1} \rightarrow \C$ are all the
roots of $p^*f_1$, then $\widetilde{\lambda}^*\alpha_1, \cdots ,
\widetilde{\lambda}^*\alpha_n$ are all the roots of $q^*(f_2)$.

\item
$\widetilde{\lambda}^*:p^*\mathcal{C}(X)[\alpha_1, \cdots , \alpha_n] \rightarrow
q^*\mathcal{C}(Y)[\widetilde{\lambda}^*\alpha_1, \cdots ,
\widetilde{\lambda}^*\alpha_n]$ is injective.

\item
The map $\widetilde{\lambda}$ induces a group monomorphism
$$\widehat{\lambda}:G_{f_2} \rightarrow G_{f_1}$$ defined by
$$\widehat{\lambda}(\phi)(\alpha)=(\widetilde{\lambda}^*)^{-1}\phi(\alpha\circ
\widetilde{\lambda})$$
\end{enumerate}
\end{proposition}

\begin{proof}
\begin{enumerate}
\item Since $f_2$ splits in $E_{f_2}$ and
$q^*f_2=q^*\lambda^*f_1=(\lambda\circ q)^*f_1$, hence $f_1$ also
splits in $E_{f_2}$. By the definition of splitting covering, there
is a covering map $\widetilde{\lambda}:E_{f_2} \rightarrow E_{f_1}$.

\item This follows from a direct computation:
$(q^*(f_2))_y((\widetilde{\lambda}^*\alpha_j)(y))=((\lambda\circ
q)^*f_1)_y(\alpha_j(\widetilde{\lambda}(y)))=((p\circ
\widetilde{\lambda})^*f_1)_y(\alpha_j(\widetilde{\lambda}(y)))=
(p^*f_1)_{\widetilde{\lambda}(y)}(\alpha_j(\widetilde{\lambda}(y)))=0$.

\item Since $\widetilde{\lambda}$ is surjective, so
$\widetilde{\lambda}^*:\mathcal{C}(E_{f_1}) \rightarrow
\mathcal{C}(E_{f_2})$ is injective. For $g\in \mathcal{C}(X)$,
$\widetilde{\lambda}^*(p^*g)=(p\circ
\widetilde{\lambda})^*g=(\lambda \circ q)^*g=q^*(\lambda^*g)\in
q^*\mathcal{C}(Y)$. Thus the restriction of $\widetilde{\lambda}^*$ to
$p^*\mathcal{C}(X)[\alpha_1, \cdots , \alpha_n]$ is also injective.

\item Observe that for $\phi \in G_{f_2}$, $\phi$ fixes $\widetilde{\lambda}^*p^*\mathcal{C}(X) \subset q^*\mathcal{C}(X)$; hence, $\phi(\widetilde{\lambda}^* (p^*\mathcal{C}(X)[\alpha_1, \cdots , \alpha_n])) \subset \widetilde{\lambda}^* (p^*\mathcal{C}(X)[\alpha_1, \cdots , \alpha_n])$. Therefore, $\widehat{\lambda}$ is well defined. It is a direct checking that $\widehat{\lambda}$ is a group
homomorphism. Suppose that
$\widehat{\lambda}(\phi_1)=\widehat{\lambda}(\phi_2)$. Then for
any $\alpha_j$,
$\phi_1(\widetilde{\lambda}^*\alpha_j)=\phi_1(\alpha_j\circ
\widetilde{\lambda})=\widetilde{\lambda}^*\widehat{\lambda}(\phi_1)(\alpha_j)=\widetilde{\lambda}^*\widehat{\lambda}(\phi_2)(\alpha_j)=
\phi_2(\widetilde{\lambda}^*\alpha_j)$. So $\phi_1=\phi_2$ and
hence $\widehat{\lambda}$ is injective.
\end{enumerate}
\end{proof}

\begin{proposition}\label{homo}
Let $f$ be a Weierstrass polynomial of degree $n$ in $X$ and split
in $Y$ where $Y\stackrel{q}{\rightarrow} X$ is a covering. Let
$\alpha_{1},\cdots ,\alpha_{n}$ be the roots of $q^*f$ in $Y$.
Suppose that $T$ is a subring of $q^{*}\mathcal{C}(X)$ and
$G=Aut_{T}T[\alpha_{1},\cdots ,\alpha_{n}]$. Then we have the
following group homomorphism $\omega_{Y,T}:A(Y/X)\rightarrow G$
defined by
$$\omega_{Y,T}(\Phi)(\beta)(y):=(\Phi^{-1})^{*}(\beta)(y):=\beta(\Phi^{-1}(y)).$$
In particular, we have a group homomorphism
$\omega_{f}=\omega_{E_{f},q^{*}\mathcal{C}(X)}:A(E_{f}/X)\rightarrow
G_{f}$ where $E_{f}\stackrel{q}{\rightarrow} X$ is the splitting
covering of $f$.
\end{proposition}
\begin{proof}
For $\Phi \in A(Y/X)$, it is easy to check that
$(\Phi^{-1})^*:T[\alpha_1, ..., \alpha_n] \rightarrow T[\alpha_1,
..., \alpha_n]$ is a ring automorphism. Since
$(\Phi^{-1})^{*}\vert_{T}=id_{T}$, $(\Phi^{-1})^*\in Aut_T
T[\alpha_1, ..., \alpha_n]$.
\end{proof}

\section{Galois correspondence}
\subsection{Correspondences between commutative rings and groups}
For a Weierstrass polynomial $f$ in $X$, we have two groups
associated to $f$: $A(E_f/X)$ and $G_f$. We will show that these two
groups are actually isomorphic. In order to do that, we use the
Galois theory of commutative rings developed by
Chase-Harrison-Rosenberg (\cite{CHR}, \cite{Gr}). All rings are
supposed to be commutative rings with identity and connected, that
is, have no idempotent other than 0 and 1 unless otherwise stated.

\begin{definition}
\begin{enumerate}
\item
A commutative $R$-algebra $S$ is separable if $S$ is a projective
$S^e$-module where $S^e=S\otimes_R S^0$ is the enveloping algebra of
$S$.

\item
Let $R$ be a subring of $S$ and $G$ be a finite subgroup of
$Aut_{R}(S)$. $S$ is said to be \textbf{G-Galois} over $R$ if
$R=S^{G}:=\lbrace x \in S : \sigma (x)=x, \; \forall \sigma \in
G \rbrace$, and there exist elements $x_{1},\cdots
,x_{n};y_{1},\cdots ,y_{n}$ of $S$ such that
$$\sum_{i=1}^{n}x_{i}\sigma(y_{i})=\delta_{e,\sigma}, \; \forall
\sigma \in G,$$ where $e$ is the identity in $G$ and
$$\delta_{e,\sigma}=\left\{
\begin{array}{cl}
1 & ,\mbox{if $\sigma = e$} \\
0 & ,\mbox{if $\sigma \neq e.$}
\end{array} \right. $$
\end{enumerate}
\end{definition}

\begin{theorem}(\cite[Theorem 2.3]{CHR})
Let S be G-Galois over R. Then there is a one-to-one
lattice-inverting correspondence between subgroups of $G$ and
separable $R$-subalgebras of $S$. If $T$ is a separable
$R$-subalgebra of $S$, then the corresponding subgroup is
$H_{T}=\lbrace \sigma \in G: \sigma (t)=t, \; \forall t \in T
\rbrace$. If $H$ is a subgroup of $G$, then the corresponding
separable $R$-subalgebra is $S^{H}=\lbrace x \in S : \sigma (x)=x,
\; \forall \sigma \in H \rbrace.$
\end{theorem}

\begin{definition}\label{Notation}
Let $\alpha_{1},\cdots ,\alpha_{n}$ be some symbols. For $\sigma \in
S_{n}$, we denote $\sum_{i_{1},\cdots ,i_{n}}
\alpha_{\sigma(1)}^{i_{1}}\alpha_{\sigma(2)}^{i_{2}}\cdots
\alpha_{\sigma(n)}^{i_{n}}$ by $\sigma(\sum_{i_{1},\cdots ,i_{n}}
\alpha_{1}^{i_{1}}\alpha_{2}^{i_{2}}\cdots \alpha_{n}^{i_{n}})$. For
$(i_{1},\cdots ,i_{n-1}),(j_{1},\cdots ,j_{n-1}) \in \N^{n-1}$, we
define $$ (i_{1},\cdots ,i_{n-1}) \prec (j_{1},\cdots ,j_{n-1})
\Leftrightarrow i_{n-k}=j_{n-k}, \; i_{n-l} < j_{n-l}, \; k =
1,\cdots ,l-1, \; for \; some \; l.$$ We denote
$\alpha^{(i_{1},\cdots
,i_{n-1})}:=\alpha_{1}^{i_{1}}\alpha_{2}^{i_{2}}\cdots
\alpha_{n-1}^{i_{n-1}}$ for $i_{k}=0,\cdots ,n-k$. Then from the
above \textquotedblleft $\prec$\textquotedblright, we give an
ordering to $A=\lbrace (i_{1},\cdots ,i_{n-1}) : i_{k}=0,\cdots
,n-k, \; k=1,\cdots ,n-1 \rbrace$. We list all elements of $A$
according to the ordering:
$$I_1=(0, 0,..., 0), I_2=(1, 0, ..., 0), I_3=(2, 0, ..., 0), ...,
I_{n!}=(n-1, n-2, ..., 1)$$ and define
$$x_j=\alpha^{I_j}, j=1, 2, ..., n!$$
Let $\sigma_{i}:=(i \; i+1 \; \cdots  \; n) \in S_{n}$ (note that
the order of $\sigma_{i}$ is $n-i+1$.) and $\sigma^{(i_{1},\cdots
,i_{n-1})}:=\sigma_{1}^{i_{1}}\sigma_{2}^{i_{2}}\cdots
\sigma_{n-1}^{i_{n-1}}$ for $i_{k}=0,\cdots ,n-k$. Similarly, we
define $\phi_{i}=\sigma^{I_i}$ for $i=1, 2, ..., n!$ Finally, we
define $V_{n}$ to be the following $n! \times n!$ matrix:
$$V_{n}:=(\phi_{i}(x_{j}))_{i,j=1,\cdots ,n!}.$$
\end{definition}

\begin{example}
For $n =2$, $V_{2}=\left( \begin{array}{cc}
1 & \alpha_{1} \\
1 & \alpha_{2}
\end{array} \right)$, and note that $det(V_{2})=\alpha_{2}-\alpha_{1}$.
\end{example}

\begin{example}
For $n=3$, $x_{1}=1$, $x_{2}=\alpha_{1}$, $x_{3}=\alpha_{1}^{2}$, $x_{4}=\alpha_{2}$,
$x_{5}=\alpha_{1}\alpha_{2}$, $x_{6}=\alpha_{1}^{2}\alpha_{2}$, $\phi_{1}=id$,
$\phi_{2}=\sigma_{1}=(1 \; 2 \; 3)$, $\phi_{3}=\sigma_{1}^{2}=(1 \; 3 \; 2)$,
$\phi_{4}=\sigma_{2}=(2 \; 3)$, $\phi_{5}=\sigma_{1}\sigma_{2}=(1 \; 2)$,
$\phi_{6}=\sigma_{1}^{2}\sigma_{2}=(1 \; 3)$, and $$V_{3}=\left( \begin{array}{cccccc}
1 & \alpha_{1} & \alpha_{1}^{2} & \alpha_{2} & \alpha_{1}\alpha_{2} & \alpha_{1}^{2}\alpha_{2} \\
1 & \alpha_{2} & \alpha_{2}^{2} & \alpha_{3} & \alpha_{2}\alpha_{3} & \alpha_{2}^{2}\alpha_{3} \\
1 & \alpha_{3} & \alpha_{3}^{2} & \alpha_{1} & \alpha_{3}\alpha_{1} & \alpha_{3}^{2}\alpha_{1} \\
1 & \alpha_{1} & \alpha_{1}^{2} & \alpha_{3} & \alpha_{1}\alpha_{3} & \alpha_{1}^{2}\alpha_{3} \\
1 & \alpha_{2} & \alpha_{2}^{2} & \alpha_{1} & \alpha_{2}\alpha_{1} & \alpha_{2}^{2}\alpha_{1} \\
1 & \alpha_{3} & \alpha_{3}^{2} & \alpha_{2} & \alpha_{3}\alpha_{2} & \alpha_{3}^{2}\alpha_{2} \\
\end{array} \right).$$
Note that \begin{align*}
det(V_{3})& =det \left( \begin{array}{cccccc}
1 & \alpha_{1} & \alpha_{1}^{2} & \alpha_{2} & \alpha_{1}\alpha_{2} & \alpha_{1}^{2}\alpha_{2} \\
1 & \alpha_{2} & \alpha_{2}^{2} & \alpha_{3} & \alpha_{2}\alpha_{3} & \alpha_{2}^{2}\alpha_{3} \\
1 & \alpha_{3} & \alpha_{3}^{2} & \alpha_{1} & \alpha_{3}\alpha_{1} & \alpha_{3}^{2}\alpha_{1} \\
1 & \alpha_{1} & \alpha_{1}^{2} & \alpha_{3} & \alpha_{1}\alpha_{3} & \alpha_{1}^{2}\alpha_{3} \\
1 & \alpha_{2} & \alpha_{2}^{2} & \alpha_{1} & \alpha_{2}\alpha_{1} & \alpha_{2}^{2}\alpha_{1} \\
1 & \alpha_{3} & \alpha_{3}^{2} & \alpha_{2} & \alpha_{3}\alpha_{2} & \alpha_{3}^{2}\alpha_{2} \\
\end{array} \right) \\
& = det \left( \begin{array}{cccccc}
1 & \alpha_{1} & \alpha_{1}^{2} & \alpha_{2} & \alpha_{1}\alpha_{2} & \alpha_{1}^{2}\alpha_{2} \\
1 & \alpha_{2} & \alpha_{2}^{2} & \alpha_{3} & \alpha_{2}\alpha_{3} & \alpha_{2}^{2}\alpha_{3} \\
1 & \alpha_{3} & \alpha_{3}^{2} & \alpha_{1} & \alpha_{3}\alpha_{1} & \alpha_{3}^{2}\alpha_{1} \\
0 & 0 & 0 & (\alpha_{3}-\alpha_{2}) & \alpha_{1}(\alpha_{3}-\alpha_{2}) & \alpha_{1}^{2}(\alpha_{3}-\alpha_{2}) \\
0 & 0 & 0 & (\alpha_{1}-\alpha_{3}) & \alpha_{2}(\alpha_{1}-\alpha_{3}) & \alpha_{2}^{2}(\alpha_{1}-\alpha_{3}) \\
0 & 0 & 0 & (\alpha_{2}-\alpha_{1}) & \alpha_{3}(\alpha_{2}-\alpha_{1}) & \alpha_{3}^{2}(\alpha_{2}-\alpha_{1}) \\
\end{array} \right) \\
& =-[(\alpha_{2}-\alpha_{1})(\alpha_{3}-\alpha_{1})(\alpha_{3}-\alpha_{2})]^{3}.
\end{align*}
\end{example}

\begin{lemma}\label{LemS}
$$S_{n}=\lbrace \sigma^{(i_{1},\cdots ,i_{n-1})} : i_{k}=0,\cdots ,n-k \rbrace .$$
\end{lemma}

\begin{proof}
Let $B_{l}=\lbrace \sigma^{(0,\cdots ,0,i_{n-l+1},\cdots ,i_{n-1})}
: i_{k}=0,\cdots ,n-k, \; for \; k=n-l+1,\cdots ,n-1 \rbrace$, for
$l=2,\cdots ,n$. We claim that $$ B_{l}= permu \lbrace n-l+1,\cdots
,n \rbrace := \lbrace \sigma \in S_{n} : \sigma(j)=j, \; \forall
j=1,\cdots ,n-l \rbrace.$$ The cases $l=1,2,3$ are obvious. Assume $
B_{l-1}= permu \lbrace n-l+2,\cdots ,n \rbrace$, $l \leqslant n$.
Then $B_{l-1} \leqslant S_{n}$, and we have a partition
$S_{n}/B_{l-1}$. Recall that for $\tau, \rho \in S_{n}$, $$\tau
B_{l-1}=\rho B_{l-1} \Leftrightarrow \tau \rho^{-1} \in B_{l-1}.$$
Also note that if $0 \leqslant p,q \leqslant l-1, \; p \neq q$, then
$(\sigma_{n-l+1}^{p})(\sigma_{n-l+1}^{q})^{-1}=\sigma_{n-l+1}^{p-q}$
is a cycle of length $l$, where $\sigma_{n-l+1}$ is defined as in
Definition~\ref{Notation}. Thus
$$(\sigma_{n-l+1}^{p})(\sigma_{n-l+1}^{q})^{-1}=\sigma_{n-l+1}^{p-q}
\in S_{n}-B_{l-1}.$$ Hence
$B_{l}=\coprod_{p=0}^{l-1}\sigma_{n-l+1}^{p}B_{l-1}$. Therefore,
$$\vert B_{l} \vert = l \vert B_{l-1} \vert = l \vert permu \lbrace
n-l+2,\cdots ,n \rbrace \vert = l!.$$ Moreover, just from the
definition, we have $B_{l} \subseteq permu \lbrace n-l+1,\cdots ,n
\rbrace$, and $\vert permu \lbrace n-l+1,\cdots ,n \rbrace \vert =
l!=\vert B_{l} \vert$. This implies $B_{l}= permu \lbrace
n-l+1,\cdots ,n \rbrace$. Thus, by induction, we proved what we
claimed. In particular, $$S_{n}=B_{n}=\lbrace \sigma^{(i_{1},\cdots
,i_{n-1})} : i_{k}=0,\cdots ,n-k , \; k=1,\cdots ,n \rbrace.$$
\end{proof}

\begin{lemma}\label{LemC}
Let $T$ be an integral domain. Suppose $\alpha_{1},\cdots
,\alpha_{n} \in T$ are distinct. Then for each $n \in \N$,
$det(V_{n}) \neq 0$.
\end{lemma}

\begin{proof}
Let $C_{k}=C_{k,0}=(\phi_{i}(x_{j}))_{i,j=1,\cdots
,\frac{n!}{(n-k)!}}$, and $C_{k,l}=(\phi_{i}(x_{j}))_{i=1,\cdots
,\frac{n!}{(n-k)!},j=l\frac{n!}{(n-k)!}+1,\cdots
,(l+1)\frac{n!}{(n-k)!}}$, for $l=0,\cdots ,(n-k-1)$. Moreover, for
$\sigma \in S_{n}$ , we define
$$C_{k,l}\sigma:=(\phi_{i}(\sigma(x_{j})))_{i=1,\cdots ,\frac{n!}{(n-k)!},j=l\frac{n!}{(n-k)!}+1,\cdots ,(l+1)\frac{n!}{(n-k)!}}.$$
Then for $i_{k+1}=0,\cdots ,n-k-1$,
\begin{align*}
C_{k,l}\sigma_{k+1}^{i_{k+1}}&=(\phi_{i}(\sigma_{k+1}^{i_{k+1}}(x_{j})))_{i=1,\cdots ,\frac{n!}{(n-k)!},j=
l\frac{n!}{(n-k)!}+1,\cdots ,(l+1)\frac{n!}{(n-k)!}}\\
&=((\sigma^{(i_{1},\cdots ,i_{k},0,\cdots ,0)}\circ \sigma_{k+1}^{i_{k+1}})(x_{j}))_{I_i=(i_{1},\cdots ,i_{k}),i=1,\cdots ,\frac{n!}{(n-k)!},
j=l\frac{n!}{(n-k)!}+1,\cdots ,(l+1)\frac{n!}{(n-k)!}}\\
&=(\sigma^{(i_{1},\cdots ,i_{k},i_{k+1},0, \cdots
,0)}(x_{j}))_{I_i=(i_{1},\cdots ,i_{k}),i=1,\cdots ,\frac{n!}{(n-k)!},
j=l\frac{n!}{(n-k)!}+1,\cdots ,(l+1)\frac{n!}{(n-k)!}}\\
&=(\phi_i(x_j))_{i=(i_{k+1})\frac{n!}{(n-k)!}+1, ..., (i_{k+1}+1)\frac{n!}{(n-k)!}, j=l\frac{n!}{(n-k)!}+1,\cdots ,(l+1)\frac{n!}{(n-k)!}}
\end{align*}
Hence, by definition, we can divide $C_{k+1}$ into several blocks in the following way:
$$C_{k+1}=\left( \begin{array}{cccc}
C_{k,0} & C_{k,1} & \ldots & C_{k,(n-k-1)} \\
C_{k,0}\sigma_{k+1} & C_{k,1}\sigma_{k+1} & \ldots & C_{k,(n-k-1)}\sigma_{k+1} \\
\vdots & \vdots & \ddots & \vdots \\
C_{k,0}\sigma_{k+1}^{n-k-1} & C_{k,1}\sigma_{k+1}^{n-k-1} & \ldots & C_{k,(n-k-1)}\sigma_{k+1}^{n-k-1} \\
\end{array} \right).$$

We want to express $det(C_{k})$ in term of $det(C_{k+1})$ and then
use induction. For $j=l\frac{n!}{(n-k)!}+1,\cdots
,(l+1)\frac{n!}{(n-k)!}$ and $x_{j}=\alpha^{(j_{1},\cdots
,j_{n-1})}$, we have $j_{k+2}=\cdots =j_{n-1}=0$ and
$$x_{j}=\alpha_{1}^{j_{1}}\cdots \alpha_{k}^{j_{k}}\alpha_{k+1}^{l}.$$
Therefore, for $p=0,1,\cdots ,n-k-1$,
$$\sigma_{k+1}^{p}(x_{j})=\alpha_{1}^{j_{1}}\cdots \alpha_{k}^{j_{k}}\alpha_{k+p+1}^{l}.$$

If $l\neq 0$,
\begin{align*}
C_{k,l}\sigma_{k+1}^{p}-C_{k,l}&=(\phi_{i}(\sigma_{k+1}^{p}(x_{j})-x_{j}))_{i=1,\cdots ,\frac{n!}{(n-k)!},j=l\frac{n!}{(n-k)!}+1,\cdots ,(l+1)\frac{n!}{(n-k)!}}\\
&=(\phi_{i}(\alpha_{1}^{j_{1}}\cdots \alpha_{k}^{j_{k}})\phi_{i}(\alpha_{k+p+1}^{l}-\alpha_{k+1}^{l}))_{i=1,\cdots ,\frac{n!}{(n-k)!},
j=l\frac{n!}{(n-k)!}+1,\cdots ,(l+1)\frac{n!}{(n-k)!}}\\
&=(\phi_{i}(\alpha_{1}^{j_{1}}\cdots\alpha_{k}^{j_{k}})\phi_{i}(\sum\limits_{q=0}^{l-1}\alpha_{k+1}^{l-1-q}\alpha_{k+p+1}^{q})
\phi_{i}(\alpha_{k+p+1}-\alpha_{k+1}))_{i=1,\cdots,\frac{n!}{(n-k)!},j=l\frac{n!}{(n-k)!}+1,\cdots,(l+1)\frac{n!}{(n-k)!}}.
\end{align*}
If $l=0$, $C_{k,0}\sigma_{k+1}^{p}-C_{k,0}=0$.

Multiplying the first row of $C_{k+1}$ by $(-1)$ and added to each
other row, we have
$$\begin{array}{ll}det(C_{k+1})&=det\left( \begin{array}{cccc}
C_{k,0} & C_{k,1} & \ldots & C_{k,(n-k-1)} \\
0 & C_{k,1}\sigma_{k+1}-C_{k,1} & \ldots & C_{k,(n-k-1)}\sigma_{k+1}-C_{k,(n-k-1)} \\
\vdots & \vdots & \ddots & \vdots \\
0 & C_{k,1}\sigma_{k+1}^{n-k-1}-C_{k,1} & \ldots & C_{k,(n-k-1)}\sigma_{k+1}^{n-k-1}-C_{k,(n-k-1)} \\
\end{array} \right) \\
\end{array}
$$
Note that $\phi_i(\alpha_{k+p+1}-\alpha_{k+1})$ is a common factor
of entries of $(p+1)$-row. Take all of them out and let the
remaining matrix be $D$, we get
$$det(C_{k+1})=
det(C_{k})det(D)(\prod\limits_{p=1}^{n-k-1}\prod\limits_{i=1}^{\frac{n!}{(n-k)!}}\phi_{i}(\alpha_{k+p+1}-\alpha_{k+1})),$$

Let the matrix
$$E^m_{r,s}
=(\phi_i(x_j)\phi_i(\sum_{q_0+\cdots+q_m=s-1}\alpha^{q_1}_{k+1}\alpha^{q_2}_{k+2}\cdots
\alpha^{q_m}_{k+m} \alpha^{q_0}_{k+r+m}))_{i, j=1, ...,
\frac{n!}{(n-k)!}}$$ and the matrix
$$D_m=(E^m_{r, s})_{r, s=1, ..., n-k-m}$$
Note that $D_1=D$ and $E^m_{r, 1}=C_k$.

Then \begin{align*} det(D_{m}) & =det \left( \begin{array}{cccc}
E^m_{1, 1} & E^m_{1, 2} & \ldots & E^m_{1, n-k-m} \\
E^{m}_{2,1} &  E^{m}_{2,2} & \ldots &  E^{m}_{2,n-k-m}\\
\vdots & \vdots & \ddots & \vdots \\
E^{m}_{n-k-m,1} & E^{m}_{n-k-m,2}& \ldots &  E^{m}_{n-k-m,n-k-m}\\
\end{array} \right)
\\ & = det \left( \begin{array}{cccc}
C_k & E^m_{1, 2} & \ldots & E^m_{1, n-k-m} \\
0 &  E^{m}_{2,2}-E^m_{1, 2} & \ldots &  E^{m}_{2,n-k-m}-E^m_{1, n-k-m}\\
\vdots & \vdots & \ddots & \vdots \\
0 & E^{m}_{n-k-m,2}-E^m_{1, 2}& \ldots &  E^{m}_{n-k-m,n-k-m}-E^m_{1, n-k-m}\\
\end{array} \right)
\end{align*}

Note that
$$E^m_{r, s}-E^m_{1, s}=(\phi_i(x)\phi_i(\sum_{q_1+\cdots
+q_m=s-1}\alpha^{q_1}_{k+1}\alpha^{q_2}_{k+2}\cdots
\alpha^{q_m}_{k+m}(\alpha^{q_0}_{k+r+m}-\alpha^{q_0}_{k+1+m}))$$
$$=(\phi_i(x)\phi_i(\sum_{q_0+\cdots
+q_m=s-1}\alpha^{q_1}_{k+1}\alpha^{q_2}_{k+2}\cdots
\alpha^{q_m}_{k+m}(\sum_{a+b=q_0-1}\\ \alpha^a_{k+m+1}\alpha^b_{k+r+m})
\phi_i(\alpha_{k+r+m}-\alpha_{k+1+m})))$$
$$=(\phi_i(x)\phi_i(\sum_{q_0+\cdots
+q_{m+1}=s-2}\alpha^{q_1}_{k+1}\alpha^{q_2}_{k+2}\cdots
\alpha^{q_m}_{k+m}\alpha^{q_{m+1}}_{k+m+1}\alpha^{q_0}_{k+(r-1)+(m+1)})
\phi_i(\alpha_{k+r+m}-\alpha_{k+1+m}))$$
$$=(E^{m+1}_{r-1, s-1}\phi_i(\alpha_{k+r+m}-\alpha_{k+1+m}))$$

So we have
$$det(D_{m})=det(C_{k})det(D_{m+1})(\prod_{r=2}^{n-k-m}\prod_{i=1}^{\frac{n!}{(n-k)!}}\phi_{i}(\alpha_{k+r+m}-\alpha_{k+m+1})).$$
Since $det(D_{n-k-1})=det(C_{k})$, inductively we get $det(D_1)\neq
0$. Therefore $det(C_{k+1})\neq 0$. Because of
$det(C_{1})=\prod_{1\leqslant i < j \leqslant
n}(\alpha_{j}-\alpha_{i})\neq 0$, by induction, $det(C_k)\neq 0$ for
$k\in \N$. In particular,
$$det(V_{n})=det(C_{n-1}) \neq 0.$$
\end{proof}

\begin{definition}
Suppose that $f$ is a Weierstrass polynomial in $X$ and $p:Y
\rightarrow X$ is a covering where $f$ splits with roots $\alpha_1,
..., \alpha_n$. Define $V:Y \rightarrow M_{n}(\C)$ by
$$V(y):=(\phi_i(x_j(y)))_{i,j=1, ..., n!}$$
and $\Delta: Y \rightarrow \C$ by
$$\Delta(y):=det(V(y))$$
where $\phi_i, x_j$ are defined in Definition ~\ref{Notation}.
\end{definition}

\begin{lemma}\label{LemGC}
Let $Y \stackrel{p}{\rightarrow} X$ be a Galois covering. Suppose
$\lambda : Y \rightarrow \C$ is a continuous map satisfying
$$\lambda \circ \Phi =\lambda , \; \forall \Phi \in A(Y/X).$$ Then
$\lambda \in p^{*}\mathcal{C}(X)$.
\end{lemma}

\begin{proof}
Since $p:Y \rightarrow X$ is a quotient map, it suffices to show
that for each $x_{0} \in X$, $x,y \in p^{-1}(x_{0})$,
$\lambda(x)=\lambda(y)$. The group $A(Y/X)$ acts on $p^{-1}(x_{0})$
transitively, so there is $\Phi \in A(Y/X)$ such that $\Phi(x)=y$.
Hence $$\lambda(y)=\lambda(\Phi(x))=\lambda(x).$$
\end{proof}

\begin{lemma}\label{RG}\label{onto}
Let $f$ be a Weierstrass polynomial in $X$ and split in
$Y\stackrel{q}{\rightarrow} X$ with roots $\alpha_{1},\cdots
,\alpha_{n}$. Let $R=q^{*}\mathcal{C}(X)$, and let $T$ be a subring
of $R$ containing coefficients of $f$ such that
\begin{enumerate}
\item $T[\alpha_{1},\cdots ,\alpha_{n}] \cap R = T$
\item $\dfrac{1}{\Delta} \in T[\alpha_{1},\cdots ,\alpha_{n}]$.
\end{enumerate} Then
\begin{enumerate}
\item
$T[\alpha_{1},\cdots ,\alpha_{n}]$ is $G$-Galois over $T$ where
$G=Aut_{T}T[\alpha_{1},\cdots ,\alpha_{n}]$.
\item
Then the group homomorphism $\omega_{Y,T}:A(Y/X)\rightarrow
G=Aut_{T}T[\alpha_{1},\cdots ,\alpha_{n}]$ is surjective.
\end{enumerate}
\end{lemma}

\begin{proof}
\begin{enumerate}
\item
We have
$$T \subset T[\alpha_{1},\cdots ,\alpha_{n}]^{G} \subset
T[\alpha_{1},\cdots ,\alpha_{n}]^{\omega_{Y, T}(A(Y/X))}.$$ By Proposition~\ref{indep}, we may assume $Y=E_{f}$ which is Galois, and hence Lemma~\ref{LemGC} and the assumptions demonstrate that $$T \subset T[\alpha_{1},\cdots
,\alpha_{n}]^{G}\subset T[\alpha_{1},\cdots ,\alpha_{n}] \cap R
=T.$$

Since $\dfrac{1}{\Delta} \in T[\alpha_{1},\cdots ,\alpha_{n}]$, we
may define
$$\left( \begin{array}{c}
y_{1}(t) \\
y_{2}(t) \\
\vdots \\
y_{n!}(t) \\
\end{array} \right) := (V(t))^{-1}\left( \begin{array}{c}
1 \\
0 \\
\vdots \\
0 \\
\end{array} \right) \mbox{ for } t\in Y$$
then
$$ V(t) \left( \begin{array}{c}
y_{1}(t) \\
y_{2}(t) \\
\vdots \\
y_{n!}(t) \\
\end{array} \right) = \left( \begin{array}{c}
1 \\
0 \\
\vdots \\
0 \\
\end{array} \right)$$ In other words,
$$\sum_{i=1}^{n!}\sigma(x_{i})y_{i}=\delta_{e,\sigma}, \; \sigma \in S_{n}$$ by Lemma~\ref{LemS}.
Since $T$ contains all coefficients of $f$, $G$ merely permutes $\alpha_1,\cdots,\alpha_n$. Therefore,
$$\sum_{i=1}^{n!}\sigma(x_{i})y_{i}=\delta_{e,\sigma}, \; \sigma \in G$$ which implies $T[\alpha_{1},\cdots ,\alpha_{n}]$
is $G$-Galois over $T$.

\item
Since $Y$ is a Galois covering over $X$, from Lemma~\ref{LemGC},
$$T \subset T[\alpha_{1},\cdots ,\alpha_{n}]^{G} \subset
T[\alpha_{1},\cdots ,\alpha_{n}]^{\omega_{Y,T}(A(Y/X))} \subset
R \cap T[\alpha_{1},\cdots ,\alpha_{n}]=T$$ which implies $$
T[\alpha_{1},\cdots ,\alpha_{n}]^{G} = T[\alpha_{1},\cdots
,\alpha_{n}]^{\omega_{Y,T}(A(Y/X))}.$$ Therefore, by
part one and Chase-Harrison-Rosenberg Theorem,
$$\omega_{Y,T}(A(Y/X))=G$$
\end{enumerate}
\end{proof}

\begin{theorem}\label{delta}
Let $f$ be a Weierstrass polynomial in $X$ and split in
$Y\stackrel{q}{\rightarrow} X$ with roots $\alpha_{1},\cdots
,\alpha_{n}$. Let $R=q^{*}\mathcal{C}(X)$. Then $R[\alpha_{1},\cdots
,\alpha_{n}]$ contains $\Delta^{-1}$. In particular, $R[\alpha_1,
\cdots , \alpha_n]$ is $G_f$-Galois over $R$ and the group homomorphisms
$$\omega_{E_{f}}:A(E_{f}/X)\rightarrow G_{f}$$ is surjective.
\end{theorem}

\begin{proof}
Since $\omega_{Y}(A(Y/X)) < G_{f}$, Lemma~\ref{LemGC} implies
$$R \subset R[\alpha_{1},\cdots ,\alpha_{n}]^{G_{f}} \subset
R[\alpha_{1},\cdots ,\alpha_{n}]^{\omega_{Y}(A(Y/X))} \subset R.$$
Hence $R[\alpha_{1},\cdots ,\alpha_{n}]^{G_{f}}=R$. By definition,
$\Delta \in R[\alpha_{1},\cdots ,\alpha_{n}]$, and for each $\sigma
\in G_{f}$ which permutes $\alpha_{1},\cdots ,\alpha_{n}$, $$\sigma(\Delta)=det((\sigma \circ
\phi_{i}(x_{j}))_{i,j=1,\cdots ,n!})=sign(\sigma)\Delta.$$
Therefore, $\Delta^{2} \in R[\alpha_{1},\cdots
,\alpha_{n}]^{G_{f}}=R$. By Lemma~\ref{LemC}, $\Delta(t) \neq 0, \;
\forall t \in Y$. Hence $\frac{1}{\Delta^{2}} \in R$, and
$\frac{1}{\Delta}=\frac{\Delta}{\Delta^{2}} \in R[\alpha_{1},\cdots
,\alpha_{n}]$. By Lemma ~\ref{RG}, $R[\alpha_1, \cdots , \alpha_n]$ is
$G_f$-Galois over $R$.
\end{proof}

\begin{corollary}(cf. Example~\ref{ExS})
Let $X=S^{1}= \lbrace z \in \C : \vert z \vert =1 \rbrace$ and $\tilde{X}=\R \stackrel{p(s)=e^{2\pi si}}{\longrightarrow} X$
is a universal covering space of $S^{1}$. Then $G_{f}$ is cyclic with a generator $\phi^{*}$, where $\phi:\R \rightarrow \R :s \mapsto s+1$.
\end{corollary}

\begin{example}(cf. Example~\ref{ExS})
Let $X=S^{1}= \lbrace z \in \C : \vert z \vert =1 \rbrace$ and $f=(z^{2}-x)(z^{2}-2x)$, $x \in S^{1}$.
Then $f$ is a Weierstrass polynomial, and its solution space $E$ is homeomorphic to a disjoint union of two circles and a 4-fold covering of $S^{1}.$

$\tilde{X}=\R \stackrel{p(s)=e^{2\pi si}}{\longrightarrow} X$  is
the universal covering space of $S^{1}$. $p^{*}f=(z^{2}-e^{2\pi
si})(z^{2}-2e^{2\pi si})$, where $s \in \R$. It is easy to see that
$$\alpha_{1}(s):=e^{\pi i s}, \; \alpha_{2}(s):=-\alpha_{1}(s), \;
\beta_{1}(s):=\sqrt{2}e^{\pi i s}, \; \beta_{2}(s)=\beta_{1}(s), \;
s \in \R$$ are all roots. By the previous corollary, $$G_{f}=<\phi^{*}>\cong
\Z_{2}.$$ Here we give another direct proof of , e.g., $\alpha_{1}
\mapsto \alpha_{2}, \; \beta_{1} \mapsto \beta_{1}$ can't induce an
element in $G_{f}$. In fact, this is immediately from the
observations $$\alpha_{1}\beta_{1} \in R=p^{*}\mathcal{C}(X),$$ but
$$\alpha_{1}\beta_{1} \neq \alpha_{2}\beta_{1}.$$
\end{example}

\subsection{The fundamental theorem of Galois theory}

In this section, we will develop a theorem corresponding to the
fundamental theorem of Galois theory. More precisely, we will
construct an 1-1 correspondences between certain ring extensions and
Galois groups, Galois groups and covering transformations, covering
transformations and covering spaces.

\begin{definition}
Let $Y \stackrel{p}{\rightarrow} X$ be a covering space and $x \in
X$. The cardinality of $p^{-1}(x)$ is called the \textbf{degree} of
$Y$ over $X$, denoted by $[Y:X]$. If $H$ is a subgroup of $G$, we
denote $\vert G/H \vert$ by $[G:H]$.
\end{definition}

The following result is clear.

\begin{lemma}
If $Z \stackrel{q}{\rightarrow} Y$ and $Y \stackrel{p}{\rightarrow} X$
are two covering spaces with finite fibres, then $Z \stackrel{pq}{\rightarrow} X$
is a covering and $$[Z:X]=[Z:Y][Y:X].$$
\end{lemma}

\begin{lemma}
If $Y \stackrel{p}{\rightarrow} X$ is a Galois covering, then $G:=A(Y/X)$ has order $[Y:X]$.
\end{lemma}

\begin{proof}
Since $Y$ is Galois over $X$, the quotient space $\pi : Y \rightarrow Y/G$ is a covering
equivalent to $Y \stackrel{p}{\rightarrow} X$. Hence the number of each fibre is $\vert G \vert$.
\end{proof}

\begin{lemma}
Let $(Z,z_{0}) \stackrel{q}{\rightarrow} (Y,y_{0})$ and $(Y,y_{0}) \stackrel{p}{\rightarrow} (X,x_{0})$ be
two covering spaces. If $Z \stackrel{pq}{\rightarrow} X$ is Galois, then $Z \stackrel{q}{\rightarrow} Y$ is Galois.
\end{lemma}

\begin{proof}
Since $Z$ is Galois over $X$, $(pq)_{*}\pi_{1}(Z,z_{0}) \lhd \pi_{1}(X,x_{0})$. Also note that
$(pq)_{*}\pi_{1}(Z,z_{0}) < p_{*}\pi_{1}(Y,y_{0})$. Hence $(pq)_{*}\pi_{1}(Z,z_{0}) \lhd p_{*}\pi_{1}(Y,y_{0})$
which implies $$q_{*}\pi_{1}(Z,z_{0}) \lhd \pi_{1}(Y,y_{0}).$$ Therefore, $Z \stackrel{q}{\rightarrow} Y$ is Galois.
\end{proof}

\begin{theorem}\label{FT}
Let $f$ be a Weierstrass polynomial of degree $n$ in $X$, and
$\alpha_1, ..., \alpha_n$ be the $n$ roots of $f$ in the splitting
covering $q:(E_f, e_0) \rightarrow (X, x_0)$. Suppose $T$ is a
subring of $R$ containing coefficients of $f$ where
$R=q^{*}\mathcal{C}(X)$ and
\begin{enumerate}
\item $T[\alpha_{1},\cdots ,\alpha_{n}] \cap R = T$
\item $\dfrac{1}{\Delta} \in T[\alpha_{1},\cdots ,\alpha_{n}]$.
\end{enumerate}
Then
\begin{enumerate}
\item $\omega=\omega_{E_{f},T}:A(E_{f}/X)\rightarrow G=Aut_{T}T[\alpha_{1},\cdots ,\alpha_{n}]$ is an isomorphism.
\item We have the following one-to-one correspondences among (based) covering spaces between $(E_{f},e_{0})\stackrel{q}{\rightarrow}
(X,x_{0})$, subgroups of $A(E_{f}/X)$, subgroups of $G_{f}$, and separable subrings of $T[\alpha_{1},\cdots ,\alpha_{n}]$ over $T$
$$\begin{array}{ccccccc}
(E_{f},e_{0}) & \longleftrightarrow & <e> & \longleftrightarrow & <e'> & \longleftrightarrow & T[\alpha_{1},\cdots ,\alpha_{n}]  \\
\downarrow & & \wedge & & \wedge & & \cup \\
(L,l_{0}) & \longleftrightarrow & H & \longleftrightarrow & H' & \longleftrightarrow & L' \\
\downarrow & & \wedge & & \wedge & & \cup \\
(M,m_{0}) & \longleftrightarrow & J & \longleftrightarrow & J' & \longleftrightarrow & M' \\
\downarrow & & \wedge & & \wedge & & \cup \\
(X,x_{0}) & \longleftrightarrow & A(E_{f}/X) & \longleftrightarrow & G & \longleftrightarrow & T
\end{array}$$
which are given by the theory of covering spaces, $\omega$, and
Chase-Harrison-Rosenberg theorem, that is, $H=A(E_{f}/L)$,
$H'=\omega (H)$, $L'=T[\alpha_{1},\cdots ,\alpha_{n}]^{H'}$, and
$H'=G_{L'}=\lbrace \phi \in G \mid \phi\vert_{L'}=id_{L'}
\rbrace$. Moreover, $$[L:M]=[J:H]=[J':H'].$$
\end{enumerate}
In particular, $$A(E_{f}/X) \cong G_{f}. $$
\end{theorem}

\begin{proof}
Let $E_{f}\stackrel{q}{\rightarrow} X$ be the splitting covering
constructed in the section of splitting coverings and $\alpha_{i}$
be the projection of the $(i+1)^{th}$ component, $i=1,\cdots ,n$.
Suppose $\Phi \in ker(\omega)$ and write $\Phi : E_{f}\rightarrow
E_{f}$ as $\Phi(x,z_{1},\cdots ,z_{n})=(\Phi_{1}(x,z_{1},\cdots
,z_{n}),\cdots ,\Phi_{n+1}(x,z_{1},\cdots ,z_{n}))$. Since
$(\Phi^{-1})^{*}(\alpha_{i})=\alpha_{i}$, $\alpha_{i} \circ \Phi =
\alpha_{i}$, $i=1,\cdots ,n$. Therefore
$$\Phi_{i+1}(x,z_{1},\cdots ,z_{n})=\alpha_{i}(\Phi(x,z_{1},\cdots ,z_{n}))=\alpha_{i}(x,z_{1},\cdots ,z_{n})=z_{i},$$
and
$$\Phi_{1}(x,z_{1},\cdots ,z_{n})=q(\Phi(x,z_{1},\cdots ,z_{n}))=q(x,z_{1},\cdots ,z_{n})=x.$$ Hence $\Phi=id_{E_{f}}$.
In other words, $\omega$ is injective. Furthermore, from Theorem~\ref{onto}, $\omega$ is surjective; therefore, $\omega$
is an isomorphism.

For the second part, the correspondences follow part one. Since $E_{f}$ is Galois over $X$, it is also Galois over $L$.
Hence $[E_{f}:L]= \vert A(E_{f}/L)= H \vert $. Similarly, $[E_{f}:M]= \vert J \vert$ Therefore,
$$[L:M]=[E_{f}:M]/[E_{f}:L]=\vert J \vert / \vert H \vert = [J:H].$$
\end{proof}

\begin{corollary}\label{conti}
Let $f$ be a Weierstrass polynomial of degree $n$ on $X$ with roots
$\alpha_{1},\cdots ,\alpha_{n}:E_{f}\rightarrow \C$ where $E_{f}
\stackrel{q}{\rightarrow}X$ is the splitting covering. Then
$$\mathcal{C}(E_{f})=
q^{*}\mathcal{C}(X)[\alpha_{1},\cdots ,\alpha_{n}].$$ In particular, $$A(E_{f}/X) \cong
Aut_{p^{*}\mathcal{C}(X)}\mathcal{C}(E_{f}).$$
\end{corollary}

\begin{proof}
    Clearly, $\mathcal{C}(E_{f})\supset q^{*}\mathcal{C}(X)[\alpha_{1},...,\alpha_{n}]$. Conversely, we consider the map $\omega :A(E_{f}/X)\rightarrow Aut_{q^{*}\mathcal{C}(X)}\mathcal{C}(E_{f})$ defined by $$\omega(\Phi)=(\Phi^{-1})^{*}.$$ By Lemma~\ref{LemGC}, $$\mathcal{C}(E_{f})^{H}=q^{*}\mathcal{C}(X)$$ where $H$ is the finite group $\omega(A(E_{f}/X))$. Moreover, from the proof of Lemma~\ref{RG}, there are $x_{1},...,x_{n!}, y_{1},...,y_{n!} \in q^{*}\mathcal{C}(X)[\alpha_{1},...,\alpha_{n}] \subset \mathcal{C}(E_{f})$ such that $$\sum_{i=1}^{n!}\sigma(x_{i})y_{i}=\delta_{e,\sigma}, \; \sigma \in H.$$
Therefore, $\mathcal{C}(E_{f})$ is $H$-Galois over $q^{*}\mathcal{C}(X)$. As a result of Chase-Harrison-Rosenberg Theorem and Lemma~\ref{RG}, $$\mathcal{C}(E_{f})= q^{*}\mathcal{C}(X)[\alpha_{1},...,\alpha_{n}]$$ since  $\mathcal{C}(E_{f})\supset q^{*}\mathcal{C}(X)[\alpha_{1},...,\alpha_{n}]$ are both separable over $q^{*}\mathcal{C}(X)$ and $$H_{q{*}\mathcal{C}(X)[\alpha_{1},...,\alpha_{n}]}=\lbrace e \rbrace =H_{\mathcal{C}(E_{f})}.$$
\end{proof}

It is natural to ask whether any finite connected Galois covering of
$X$ (that is, connected Galois covering of finite degree) is
equivalent to a splitting covering of a Weierstrass polynomial. At
this moment we have not yet solved this problem, so we make it as a
conjecture.

\begin{conjecture}\label{conj}
Suppose $Y \stackrel{p}{\rightarrow}X$ is a finite connected Galois
covering of $X$. Then there exists a Weierstrass polynomial $f$ on
$X$ such that the splitting covering
$E_{f}\stackrel{q}{\rightarrow}X$ is equivalent to $Y
\stackrel{p}{\rightarrow}X$.
\end{conjecture}

One reason to believe that this conjecture is true is due to the
following result.

\begin{theorem}\label{quote} (\cite[Theorem 6.3, pg 110]{Ha})
Suppose the $\pi_{1}(X)$ is a free group. Then every finite covering map onto $X$
is equivalent to a polynomial covering map, that is,
there is a Weierstrass polynomial on $X$ such that its solution space is equivalent to that
finite covering as a covering space.
\end{theorem}

We have the following partial result.

\begin{proposition}\label{SCC}
Suppose $\pi_{1}(X)$ is free. Then any finite connected Galois covering of $X$ is equivalent to a
splitting covering of a Weierstrass polynomial on $X$.
\end{proposition}

\begin{proof}
Let $Y \stackrel{p}{\rightarrow}X$ be a finite connected Galois
covering of $X$. Theorem~\ref{quote} demonstrates that there is a
Weierstrass polynomial $f$ on $X$ such that its solution space is
equivalent to $Y \stackrel{p}{\rightarrow}X$. Moreover, by
Corollary~\ref{CorSplit}, the splitting covering $E_{f}
\stackrel{q}{\rightarrow}X$ is equivalent to $Y
\stackrel{p}{\rightarrow}X$.
\end{proof}

\begin{corollary}
Suppose $\pi_{1}(X)$ is free and $Y \stackrel{p}{\rightarrow}X$ is a
finite connected Galois covering of $X$. Then $$A(Y/X) \cong
Aut_{p^{*}\mathcal{C}(X)}\mathcal{C}(Y),$$
and $\mathcal{C}(Y)$ is
$A(Y/X)$-Galois over $p^{*}\mathcal{C}(X)$.
\end{corollary}

\begin{proof}
By Proposition~\ref{SCC}, there exists a Weierstrass polynomial $f$
on $X$ with splitting covering $E_{f} \stackrel{q}{\rightarrow}X$
equivalent to $Y \stackrel{p}{\rightarrow}X$, that is, there is a
covering equivalence $\Phi :Y \rightarrow E_{f}$. Therefore,
$\Phi^{*}$ gives an isomorphism from $\mathcal{C}(E_{f})$ to
$\mathcal{C}(Y)$ mapping $q^{*}\mathcal{C}(X)$ onto
$p^{*}\mathcal{C}(X)$. As a result of Proposition~\ref{conti}, $$A(Y/X)\cong A(E_{f}/X) \cong Aut_{p^{*}\mathcal{C}(X)}\mathcal{C}(Y),$$
and $\mathcal{C}(Y)$ is $A(Y/X)$-Galois over $p^{*}\mathcal{C}(X)$.
\end{proof}

\section{Groups as Galois groups}
The inverse Galois problem is a hundred-year old open problem which
asks if every finite group can be realized as some Galois group over
the field of rational numbers. Hilbert used his irreducibility
theorem to realize some groups like $S_n$ and Shafarevich used tools
from number theory to show that every solvable finite group can be
realized. Recently rigidity method is used to realize many finite
groups but there is no way which may realize all finite groups.
Here we provide a new viewpoint from our semi-topological Galois
theory toward the inverse Galois problem.

\begin{proposition}
Let $G$ be any finite group. For any compact disc $D\subseteq \C$,
there exist some disjoint open discs $D_1, ..., D_m\subseteq D$ and
an irreducible Weierstrass polynomial defined on $X=D -
\cup^m_{j=1}D_j$ with semi-topological Galois group $G$.
\end{proposition}

\begin{proof}
Since $G$ is finite, there exists a finite generated free group $F$
and a normal subgroup $N$ of $F$ such that $$G=F/N.$$ If $F$ is
generated by $m$ elements, then we take $m$ disjoint open discs
$D_1, ..., D_m$ in $D$ such that
$$\pi_{1}(X) \cong F$$ where $X=D-\cup^m_{j=1}D_j$. By the Galois correspondence of covering
spaces, there is a connected covering space
$E\stackrel{p}{\rightarrow} X$ such that
$$N \cong p_{*}\pi_{1}(E) $$ and $$A(E/X) \cong F/N \cong G.$$ Hence
$E \stackrel{p}{\rightarrow} X$ is Galois. By Theorem~\ref{quote},
$E$ is equivalent to the solution space $E'$ of some irreducible
Weierstrass polynomial $f$ defined on $X$. Since $p':E' \rightarrow
X$ is Galois, by Corollary~\ref{CorSplit}, $E'$ is the splitting
covering of $f$. By Theorem~\ref{FT}, the semi-topological Galois
group of $f$ is $G$.
\end{proof}

\begin{corollary}\label{CorIG}
Let $G$ be any finite group and $D$ be a compact disc in $\C$. Then
there exist finite disjoint open discs $D_{1},\cdots ,D_m$ in $D$
and an irreducible Weierstrass polynomial
in $X=D-\cup_{j=1}^{m}D_{j}$ with coefficients in $\Q(i)[x]$ and
semi-topological Galois group $G$.
\end{corollary}

\begin{proof}
By the above result, there is a Weierstrass polynomial
$f=z^{n}+a_{n-1}z^{n-1}+\cdots +a_{0} \in \mathcal{C}(X)[z]$ such
that $G_{f} \cong G$ for some $X=D-\cup_{j=1}^{N}D_{j}$. Let $a:X
\rightarrow B^n$ be the continuous function defined by
$$a(x)=(a_{0}(x),\cdots ,a_{n-1}(x))$$ Then $a(X)\subset B^{n}$
is compact where $B^n=\C^n-Z(\delta)$ and $\delta$ is the
discriminant polynomial. Since $Z(\delta)$ is closed in $\C^{n}$,
the distance between $a(X)$ and $Z(\delta)$ is $\varepsilon =
d(a(X),Z(\delta))>0$.

By the Stone-Weierstrass theorem, there are $\tilde{a}_{0},\cdots
,\tilde{a}_{n-1} \in \Q(i)[x]$ such that
$$\Vert \tilde{a}_{j}-a_{j}\Vert <  \varepsilon/2n  ,  \; j=0,\cdots ,n-1$$
where $||\tilde{a}_j-a_j||=max_{x \in X} |\tilde{a}_j(x)-a_j(x)|$.
Hence $$||\tilde{a}-a||\leq
\sum^n_{i=1}||\tilde{a}_j-a_j||<\varepsilon/2.$$ Then for any $x\in
X$,
$$d(\tilde{a}(x), Z(\delta))\geq d(a(x), Z(\delta))-d(a(x),
\tilde{a}(x))>\varepsilon-\varepsilon/2=\varepsilon/2.$$

Therefore we have a map $\tilde{a}=(\tilde{a}_{0},\cdots
,\tilde{a}_{n-1}):X \rightarrow B^{n}$ and a Weierstrass polynomial
$\tilde{f}=z^{n}+ \sum_{j=0}^{n-1}\tilde{a}_{j}z^{j}$.

Let
$$H(x,t):=(1-t)a(x)+t\tilde{a}(x)$$ for $t\in [0, 1], x\in X$.
Then $|a(x)-H(x, t)|=t|a(x)-\tilde{a}(x)|<t\varepsilon/2\leq
\varepsilon/2$, so $H:X\times I \rightarrow B^n$ is a homotopy
between $a$ and $\tilde{a}$. Hence Corollary~\ref{homotopy} and
Theorem~\ref{FT} imply that
$$G_{\tilde{f}} \cong G_{f} \cong G.$$
\end{proof}

In the following, we fix $X$ a path-connected subset of $\C$. Let
$E_{f}\stackrel{q}{\rightarrow} X$ be the splitting covering of a
Weierstrass polynomial $f$, and let $\alpha_{1},\cdots ,\alpha_{n}
:E_{f}\rightarrow \C$ be all roots of $f$. We define $$T:=\lbrace
\frac{q^{*}g}{q^{*}h} \mid g, h \in \Q(i)[x], \; h(x) \neq 0 , \;
\forall x \in X \rbrace$$ and $$W:=\lbrace \frac{q^{*}g}{q^{*}h}
\mid g, h \in \Q(i)[x], \; h \neq 0 \rbrace \cong \Q(i)(x).$$

\begin{lemma}\label{CQ}
$$Aut_{T}T[\alpha_{1},\cdots ,\alpha_{n}] \cong Aut_{W}W[\alpha_{1},\cdots ,\alpha_{n}].$$
\end{lemma}

\begin{proof}
Since $T$ is a subring of $W$, we have a restriction map
$r:Aut_{W}W[\alpha_{1},\cdots ,\alpha_{n}]\rightarrow
Aut_{T}T[\alpha_{1},\cdots ,\alpha_{n}]$ by sending $\varphi$ to
$\varphi \vert_{T[\alpha_{1},\cdots ,\alpha_{n}]}$. We use that
notation $\alpha^I:=\alpha_1^{i_1}\alpha_2^{i_2}\cdots
\alpha_n^{i_n}$ where $I=(i_1, i_2, ..., i_n)$.

For $\psi \in Aut_{T}T[\alpha_{1},\cdots ,\alpha_{n}]$, define
$\widetilde{\psi}:W[\alpha_{1},\cdots ,\alpha_{n}] \rightarrow
W[\alpha_{1},\cdots ,\alpha_{n}]$ by
$$\widetilde{\psi}(\sum_I\frac{a_I}{b_I}\alpha^I)=\frac{1}{B}\psi(\sum_Ia_IB_I\alpha^I)$$
where $B=\prod_Ib_I$ and $B_I=\frac{B}{b_I}$.

\begin{description}
\item[Well defined:] If $\sum_I\frac{a_I}{b_I}\alpha^I=0$, then
$\sum_Ia_IB_I\alpha^I=0$. So
$\frac{1}{B}\psi(\sum_Ia_IB_I\alpha^I)=0$ which implies that
$\widetilde{\psi}(\sum_I\frac{a_I}{b_I}\alpha^I)=0$.

\item[Homomorphism:]
$\widetilde{\psi}((\sum_I\frac{a_I}{b_I}\alpha^I)(\sum_J\frac{c_J}{d_J}\alpha^J))
=
\frac{1}{BD}\psi((\sum_Ia_IB_I\alpha^I)(\sum_Jc_JD_J\alpha^J))=$
\\ $ [\frac{1}{B}\psi(\sum_Ia_IB_I \alpha^I)]
[\frac{1}{D}\psi(\sum_Jc_JD_J\alpha^J)]
=\widetilde{\psi}(\sum_I\frac{a_I}{b_I}\alpha^I)\widetilde{\psi}
(\sum_J\frac{c_J}{d_J}\alpha^J)$ where $D=\prod_Jc_J$ and
$D_J=\frac{D}{c_J}$.

\item[Injective:]
If $\widetilde{\psi}(\sum_I\frac{a_I}{b_I}\alpha^I)=0$, then
$\frac{1}{B}\psi(\sum_Ia_IB_I\alpha^I)=0$, so
$\psi(\sum_Ia_IB_I\alpha^I)=0$. The map $\psi$ is an
automorphism, so $\sum_Ia_IB_I\alpha^I=0$ and hence
$\sum_I\frac{a_I}{b_I}\alpha^I=0$.

\item[Surjective:]
If $\sum_I\frac{a_I}{b_I}\alpha^I\in W[\alpha_1, ...,
\alpha_n]$, then
$\widetilde{\psi}(\sum_I\frac{a_I}{b_I}(\psi^{-1}(\alpha))^I)=\sum_I\frac{a_I}{b_I}\alpha^I$.
\end{description}

Consequently, $\widetilde{\psi} \in Aut_{W}W[\alpha_{1},\cdots
,\alpha_{n}]$.

We obtain a map $\Phi :Aut_{T}T[\alpha_{1},\cdots ,\alpha_{n}]
\rightarrow Aut_{W}W[\alpha_{1},\cdots ,\alpha_{n}]$ defined by
$$\Phi(\psi)=\widetilde{\psi}.$$ Since
$$\Phi(\psi_1\circ
\psi_2)(\sum_I\frac{a_I}{b_I}\alpha^I)=\frac{1}{B}\psi_1\circ
\psi_2(\sum_Ia_IB_I\alpha^I)
=\frac{1}{B}\psi_1(\sum_Ia_IB_I\psi_2(\alpha^I))$$
$$=\widetilde{\psi}_1(\sum_I\frac{a_I}{b_I}\psi_2(\alpha^I))
=\widetilde{\psi}_1(\frac{1}{B}\psi_2(\sum_Ia_IB_I\alpha^I))=\widetilde{\psi}_1\circ
\widetilde{\psi}_2(\sum_I\frac{a_I}{b_I}\alpha^I)=\Phi(\psi_1)\Phi(\psi_2)(\sum_I\frac{a_I}{b_I}\alpha^I),
$$ $\Phi$ is a
group homomorphism. It is clear that $\Phi \circ r = id$ and $r
\circ \Phi = id$, so $\Phi$ is a group isomorphism.
\end{proof}

\begin{proposition}
Suppose $T[\alpha_{1},\cdots ,\alpha_{n}] \cap R=T$ where
$R=q^{*}\mathcal{C}(X)$. Then $G_{f}$ occurs as a Galois group of a
Galois extension of $\Q(i)$.
\end{proposition}
\begin{proof}
Since $\Delta^2\in R$ and $\Delta\in T[\alpha_1, ..., \alpha_n]$, so
$\Delta^2\in T$ by assumption. So $\Delta^2=\frac{q^*g}{q^*h}$ for
some $g, h\in \Q(i)[x]$ and $h(x)\neq 0$ for $x\in X$. By Lemma
\ref{LemC}, $\frac{1}{\Delta^2}=\frac{q^*h}{q^*g}\in R$, so
$g(x)\neq 0$ for all $x\in X$. This implies $\frac{1}{\Delta^2}\in
T$. Then $\frac{1}{\Delta}=\frac{\Delta}{\Delta^2}\in T[\alpha_1,
..., \alpha_n]$. By Theorem \ref{FT},
$$Aut_{T}(T[\alpha_{1},\cdots,\alpha_{n}]) \cong A(E/X) \cong G_{f}$$
Lemma \ref{CQ} implies that
$$Aut_{W}(W[\alpha_{1},\cdots ,\alpha_{n}]) \cong G_{f}.$$
Moreover, $W[\alpha_{1},\cdots ,\alpha_{n}]$ is the splitting field
of $f \in W[z]$, and hence, $W[\alpha_{1},\cdots ,\alpha_{n}]$ is a
Galois extension of $W$. Since $W \cong \Q(i)(x)$ and $\Q(i)$ is
Hilbertian, $G_{f}$ occurs as a Galois group of certain Galois
extension $L$ of $\Q(i)$ as wished (see \cite{Vo}).
\end{proof}

\begin{acknowledgement}
The authors thank the National Center of Theoretical Sciences of
Taiwan (Hsinchu) for providing a wonderful working environment.
\end{acknowledgement}

\bibliographystyle{amsplain}

Authors' addresses:

Hsuan-Yi Liao, Department of Mathematics, National Tsing Hua
University of Taiwan, Hsinchu, 30043, Taiwan. Email:
s9821502@m98.nthu.edu.tw

Jyh-Haur Teh, Department of Mathematics, National Tsing Hua
University of Taiwan, Hsinchu, 30043, Taiwan. Email:
jyhhaur@math.nthu.edu.tw

\end{document}